\documentclass[12pt]{article}
\usepackage{mycommands}
\usepackage{tikz}
\title{The Maximal Rank Conjecture for Sections of Curves}

\begin{document}
\maketitle

\begin{abstract}
Let $C \subset \pp^r$ be a general curve of genus $g$
embedded via a general linear series of degree $d$.
The \emph{Maximal Rank Conjecture} asserts that
the restriction maps $H^0(\oo_{\pp^r}(m)) \to H^0(\oo_C(m))$
are of maximal rank; this determines
the Hilbert function of $C$.

In this paper, we prove an analogous statement
for the union of hyperplane sections of general curves.
More specifically, if $H \subset \pp^r$ is a general hyperplane,
and $C_1, C_2, \ldots, C_n$ are general curves,
we show $H^0(\oo_{H}(m)) \to H^0(\oo_{(C_1 \cup C_2 \cup \cdots \cup C_n) \cap H}(m))$
is of maximal rank, except for some counterexamples when $m = 2$.

As explained in \cite{over}, this result plays a key role in the author's proof
of the Maximal Rank Conjecture \cite{mrc}.
\end{abstract}

\section{Introduction}

Let $\mathcal{H}_{d, g, r}$ denote the Hilbert scheme
classifying subschemes of $\pp^r$ with Hilbert polynomial $P(x) = dx + 1 - g$.
We have a natural rational map
from any component of $\mathcal{H}_{d, g, r}$ whose general
member is a smooth curve to the moduli space $M_g$ of curves.
The Brill--Noether theorem
asserts that there exists such a component whose general member is nondegenerate and that dominates $M_g$
if and only if
\[\rho(d, g, r) := (r + 1)d - rg - r(r + 1) \geq 0.\]
Moreover, it is known that when $\rho(d, g, r) \geq 0$,
there exists a unique such component that dominates $M_g$.
We shall refer to a curve $C \subset \pp^r$
lying in this component as a
\emph{Brill--Noether Curve (BN-curve)}.

A natural first step in understanding the extrinsic
geometry of general curves is to understand their
Hilbert function. Here we have the
\emph{Maximal Rank Conjecture}:

\begin{conj}[Maximal Rank Conjecture]
If $C$ is a general BN-curve and $m$ is a positive integer,
then the restriction map
\[H^0(\oo_{\pp^r}(m)) \to H^0(\oo_C(m))\]
is of maximal rank.
\end{conj}

\begin{rem} Since $H^1(\oo_C(m)) = 0$ for $m \geq 2$
when $C$ is a general BN-curve,
the Maximal Rank Conjecture completely determines the
Hilbert function of $C$.
\end{rem}

In this paper, we study a related question for hyperplane sections.
Namely, we prove that the general hyperplane section
of a general union of BN-curves imposes the expected number of conditions
on hypersurfaces of every degree,
apart from a few counterexamples that occur for quadric
hypersurfaces.
Both the results and the techniques
developed here play a critical role in the author's proof of the Maximal Rank Conjecture \cite{mrc},
as explained in \cite{over}.
More precisely, we prove:

\begin{thm}[Hyperplane Maximal Rank Theorem] \label{main}
If $C_1, C_2, \ldots, C_n$ are independently general BN-curves
of degrees $d_i$ and genera $g_i$, and
$H \subset \pp^r$ is a general hyperplane,
and $m$ is a positive integer,
then the restriction map
\[H^0(\oo_H(m)) \to H^0(\oo_{(C_1 \cup C_2 \cup \cdots \cup C_n) \cap H}(m))\]
is of maximal rank, except possibly when
$m = 2$ and $d_i < g_i + r$ for some $i$.
\end{thm}

The conclusion that this restriction map
is of maximal rank can be reformulated in
terms of the cohomology of the twists of the ideal sheaf
as follows:
\begin{align*}
H^0(\mathcal{I}_{((C_1 \cup C_2 \cup \cdots \cup C_n) \cap H) / H}(m)) = 0 &\quad \text{when}\ \sum_{i = 1}^n d_i \geq \binom{m + r - 1}{r - 1}, \\
H^1(\mathcal{I}_{((C_1 \cup C_2 \cup \cdots \cup C_n) \cap H) / H}(m)) = 0 &\quad \text{when}\ \sum_{i = 1}^n d_i \leq \binom{m + r - 1}{r - 1}.
\end{align*}

In the course of proving Theorem~\ref{main},
we will also prove stronger results for $r = 3$ and for $r = 4$.
Namely:

\begin{thm} \label{add3}
Let $X \subset H \simeq \pp^2 \subset \pp^3$ be a
subscheme,
and $C \subset \pp^3$ be a general BN-curve.
\begin{itemize}
\item 
If $C$ is a canonical curve and $m = 2$,
suppose that $X$ is nonempty.
\item If $C$ is a canonical curve and $m \neq 2$,
write $\Lambda \subset H$ for a general line, and
suppose that the restriction maps
\begin{gather*}
H^0(\oo_H(m)) \to H^0(\oo_X(m)) \\
H^0(\oo_H(m - 1)) \to H^0(\oo_X(m - 1)) \\
H^0(\oo_\Lambda(m)) \to H^0(\oo_{X \cap \Lambda}(m))
\end{gather*}
are of maximal rank, with either the second one an injection,
or the third one a surjection with kernel of dimension at least $4$.
\item Otherwise, suppose the map
\[H^0(\oo_H(m)) \to H^0(\oo_{X}(m))\]
is of maximal rank.
\end{itemize}
Then the map
\[H^0(\oo_H(m)) \to H^0(\oo_{X \cup (C \cap H)}(m))\]
is of maximal rank.
\end{thm}

\begin{thm} \label{add4}
Let $X \subset H \simeq \pp^3 \subset \pp^4$ be a
subscheme,
and $C \subset \pp^4$ be a general BN-curve of degree $d$ and genus $g$.
\begin{itemize}
\item If $(d, g) \in \{(8, 5), (9, 6), (10, 7)\}$ and $m = 2$,
suppose that $X$ is either positive dimensional or of degree at least $11 - d$.
\item If $(d, g) \in \{(8, 5), (9, 6), (10, 7)\}$ and $m \neq 2$,
write $\Lambda \subset H$ for a general plane, and
suppose that the restriction maps
\begin{gather*}
H^0(\oo_H(m)) \to H^0(\oo_X(m)) \\
H^0(\oo_H(m - 1)) \to H^0(\oo_X(m - 1)) \\
H^0(\oo_\Lambda(m)) \to H^0(\oo_{X \cap \Lambda}(m))
\end{gather*}
are of maximal rank, with either the second one an injection,
or the third one a surjection with kernel of dimension at least $8$.
\item Otherwise, suppose the map
\[H^0(\oo_H(m)) \to H^0(\oo_{X}(m))\]
is of maximal rank.
\end{itemize}
Then the map
\[H^0(\oo_H(m)) \to H^0(\oo_{X \cup (C \cap H)}(m))\]
is of maximal rank.
\end{thm}

We shall prove Theorem~\ref{main} using an inductive
approach due originally to Hirschowitz \cite{mrat}. In its simplest form,
suppose that $C = X \cup Y$ is a reducible curve
such that $Y$ is contained in some hyperplane $H'$:

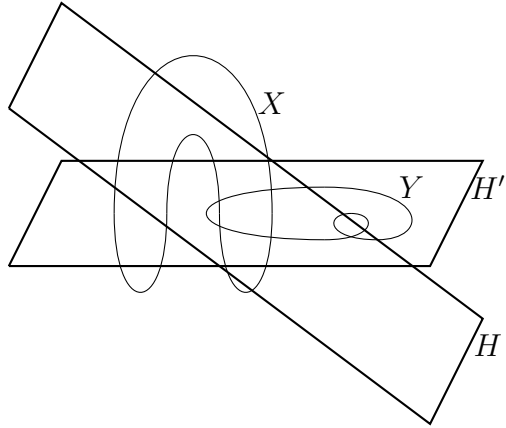
\begin{center}
\begin{tikzpicture}[scale=0.7]
\draw[thick] (0, 1) -- (1, 3) -- (9, 3) -- (8, 1) -- (0, 1);
\draw[thick] (0, 4) -- (1, 6) -- (9, 0) -- (8, -2) -- (0, 4);
\draw (2, 2) .. controls (2, 6) and (5, 6) .. (5, 2);
\draw (2, 2) .. controls (2, 0) and (3, 0) .. (3, 2);
\draw (4, 2) .. controls (4, 0) and (5, 0) .. (5, 2);
\draw (3, 2) .. controls (3, 4) and (4, 4) .. (4, 2);
\draw (6, 2.5) .. controls (3, 2.5) and (3, 1.5) .. (6, 1.5);
\draw (6, 2.5) .. controls (8, 2.5) and (8, 1.5) .. (7, 1.5);
\draw (6, 1.5) .. controls (7, 1.5) and (7, 2) .. (6.5, 2);
\draw (7, 1.5) .. controls (6, 1.5) and (6, 2) .. (6.5, 2);
\draw (9.1, 2.5) node{$H'$};
\draw (9.1, -0.5) node{$H$};
\draw (5, 4.1) node{$X$};
\draw (7.65, 2.5) node{$Y$};
\end{tikzpicture}
\end{center}

\noindent
Then we have the exact sequence of sheaves
\[0 \to \mathcal{I}_{(X \cap H) / H}(m - 1) \to \mathcal{I}_{(C \cap H) / H}(m) \to \mathcal{I}_{(Y \cap H)/(H \cap H')}(m) \to 0,\]
which gives rise to a long exact sequence in cohomology
\[\cdots \to H^i(\mathcal{I}_{(X \cap H) / H}(m - 1)) \to H^i(\mathcal{I}_{(C \cap H) / H}(m)) \to H^i(\mathcal{I}_{(Y \cap H)/(H \cap H')}(m)) \to \cdots.\]
Consequently, we can deduce the hyperplane maximal rank theorem for the
general hyperplane section of $C$ from the hyperplane maximal rank theorem
for the general hyperplane sections of $X$ and $Y$.

The structure of this paper is as follows.
First, in Section~\ref{defc}, we give several
methods of constructing reducible BN-curves
that will be useful for specialization
arguments later on.
In Sections~\ref{r} and~\ref{m}, we prove the
hyperplane maximal rank theorem in the special cases
$r = 3$ and $m = 2$ respectively.
We then deduce the general case in Sections~\ref{sec:glue}
and~\ref{sec:ind} via the above inductive argument,
by finding appropriate BN-curves
$X \subset \pp^r$
and $Y \subset H' \subset \pp^r$
satisfying the hyperplane maximal rank theorem for $(m - 1, r)$ and
$(m, r - 1)$ respectively.

\medskip

\textit{Notational Convention}: We say a BN-curve $X \subset \pp^r$ is \emph{nonspecial}
if $d \geq g + r$, i.e.\ if $X$ is a \emph{limit} of curves with nonspecial
hyperplane section.

\subsection*{Acknowledgements}
I would like to thank Professor Joe Harris for valuable comments
and discussions. This research was supported by the Harvard PRISE
and Herchel-Smith fellowships.

\section{\label{defc} Some Gluing Lemmas}

In this section, we will give some lemmas that
let us construct examples of BN-curves.

\begin{lm} \label{glone}
Let $X \subset \pp^r$ be a curve with $H^1(N_X) = 0$,
and $D$ be a rational normal curve of degree $d \leq r$ that
is $k$-secant to $X$, where
\[k \leq \begin{cases}
d + 1 & \text{if $d < r$;} \\
r + 2 & \text{if $d = r$.}
\end{cases}\]
Then $X \cup D$ is smoothable and $H^1(N_{X \cup D}) = 0$.
Moreover, if $X$ is a BN-curve, then $X \cup D$ is a BN-curve.
\end{lm}
\begin{proof}
The vanishing of $H^1(N_{X \cup D})$ and smoothability
of $X \cup D$ are consequences of Theorem~4.1 of \cite{hh}
(via the same argument as Corollary~4.2 of \cite{hh}),
together with the fact that
\[N_D = \oo_{\pp^1}(d)^{\oplus(r - d)} \oplus \oo_{\pp^1}(d + 2)^{\oplus (d - 1)}.\]

Now assume $X$ is a BN-curve.
To show that $X \cup D$ is a BN-curve, we just need to count
the dimension of the space of embeddings of $X \cup D$ into projective
space (this suffices because there is a unique component of the
Hilbert scheme that dominates $M_g$).
In order to do this, first note that
\[\rho(X \cup D) = \rho(X) + (r + 1)d - r(k - 1).\]
Consequently, the verification
that $X \cup D$ is a BN-curve boils down to the following two assertions,
both of which are straight-forward to check:
\begin{enumerate}
\item Given a $\pp^1$ with $k \leq d + 1$ marked points,
the family of degree $d$ embeddings of $\pp^1$ as a rational
normal curve with given values at the marked points
has dimension
\[(r - d)(d - k + 1) + d(d + 2 - k) = (r + 1)d - r(k - 1).\]
\item Given a $\pp^1$ with $r + 2$ marked points,
there is a unique embedding of $\pp^1$ as a rational normal
curve of degree $r$ with given values at all marked points.
\end{enumerate}
This completes the proof.
\end{proof}

\begin{lm} \label{glonered}
Let $X \subset \pp^r$ be a curve with $H^1(N_X) = 0$,
and $R$ be a rational normal curve of degree $r - 1$
that is $(r + 1)$-secant to $X$, and $L$ be a line
that is $1$-secant to both $X$ and $R$.
Then $H^1(N_{X \cup R \cup L}) = 0$.
\end{lm}
\begin{proof}
Note that for curves $A$ and $B$,
\[ H^1(N_{A \cup B}|_A) = 0 \tand H^1(N_{A \cup B}|_B(-A \cap B)) = 0 \quad \Rightarrow \quad H^1(N_{A \cup B}) = 0;\]
indeed, this holds for $N_{A \cup B}$ replaced
by any vector bundle.

In particular, since
$N_A$ is a subbundle of full rank in $N_{A \cup B}|_A$,
we can conclude that $H^1(N_{A \cup B}) = 0$
provided that
\begin{align*}
H^1(N_A) = 0 &\tand H^1(N_{A \cup B}|_B(-A \cap B)) = 0, \\
\text{or respectively} \quad H^1(N_{A \cup B}|_A) = 0 &\tand H^1(N_B(-A \cap B)) = 0.
\end{align*}
Thus, the vanishing of $H^1(N_{X \cup R \cup L})$ follows from
the following facts:
\begin{align*}
H^1(N_X) &= 0 \\[-0.5ex]
H^1(N_{R \cup L}|_R(-X \cap R)) &= H^1(\oo_{\pp^1}^{\oplus (r - 2)} \oplus \oo_{\pp^1}(-1)) = 0. \\
H_1(N_L(-L \cap (X \cup R))) &= H^1(\oo_{\pp^1}(-1)) = 0. \qedhere
\end{align*}
\end{proof}

\begin{lm} \label{gltwo}
Let $X \subset \pp^r$ be a curve with $H^1(N_X) = 0$,
and $L$ be a line $3$-secant to $X$.
Assume that the tangent lines to $X$ at the three points
of intersection do not all lie in a plane.
Then $X \cup D$ is smoothable and $H^1(N_{X \cup D}) = 0$.
\end{lm}
\begin{proof}
See Remark 4.2.2 of \cite{hh}.
\end{proof}

%
%

We end this section with two simple observations,
that will be used several times in the remainder of the paper
and will therefore be useful to spell out.

\begin{lm} \label{obvious}
Let $\mathcal{X}$ and $\mathcal{Y}$ be 
irreducible families of curves in $\pp^r$, sweeping out subvarieties
$\bar{\mathcal{X}}, \bar{\mathcal{Y}} \subset \pp^r$
of codimension at most one.
Let $X$ and $Y$ be
specializations of $\mathcal{X}$ and $\mathcal{Y}$ respectively,
such that $X \cup Y$
is a BN-curve with $H^1(N_{X \cup Y}) = 0$, and $X \cap Y$
is quasi-transverse and general in $\bar{\mathcal{X}} \cap \bar{\mathcal{Y}}$.

Then there are simultaneous generalizations $X'$ and $Y'$ of
$X$ and $Y$ respectively
such that $X' \cup Y'$ is a BN-curve with $\#(X \cap Y) = \#(X' \cap Y')$.
Equivalently, in more precise language, write
$B_1$ and $B_2$ for the bases of $\mathcal{X}$ and $\mathcal{Y}$
respectively. Then
we are asserting the existence of an irreducible $B \subset B_1 \times B_2$
dominating both $B_1$ and $B_2$, such that any fiber
$(X', Y')$ of $(\mathcal{X} \times \mathcal{Y}) \times_{(B_1 \times B_2)} B$
satisfies the given conclusion.
\end{lm}
\begin{proof}
As $\bar{\mathcal{Y}}$ has codimension at most one,
the intersection of any generalization $X'$
of $X$ with $\bar{\mathcal{X}} \cap \bar{\mathcal{Y}}$
contains a generalization of $X \cap Y$.
Similarly, the intersection of any generalization
$Y'$ of $Y$ with $\bar{\mathcal{X}} \cap \bar{\mathcal{Y}}$
contains a generalization of $X \cap Y$.
The existence of simultaneous generalizations $X'$ and $Y'$ of
$X$ and $Y$ respectively
with $\#(X \cap Y) = \#(X' \cap Y')$
thus follows from the generality of $X \cap Y$
in $\bar{\mathcal{X}} \cap \bar{\mathcal{Y}}$.

Moreover, since $H^1(N_{X \cup Y}) = 0$, the curve $X \cup Y$ is a smooth
point of the corresponding Hilbert scheme; consequently, any
generalization $X' \cup Y'$ of $X \cup Y$ is a BN-curve.
\end{proof}

\begin{lm} \label{addsub} Let $S \subset \pp^r$ and $T \subset \pp^r$ be sets
of points such that the restriction maps
\[H^0(\oo_{\pp^r}(m)) \to H^0(\oo_S(m)) \tand H^0(\oo_{\pp^r}(m)) \to H^0(\oo_{S \cup T}(m))\]
are of maximal rank. Then, for every integer $0 \leq n \leq \# T$,
there exists a subset $T' \subset T$ of cardinality $n$ such that
\[H^0(\oo_{\pp^r}(m)) \to H^0(\oo_{S \cup T'}(m))\]
is of maximal rank.

In particular, taking $T = \pp^r(\cc) \smallsetminus S$,
if $H^0(\oo_{\pp^r}(m)) \to H^0(\oo_S(m))$
is of maximal rank, then for $n$ general points $T' \subset \pp^r$,
the map $H^0(\oo_{\pp^r}(m)) \to H^0(\oo_{S \cup T'}(m))$
is also of maximal rank.
\end{lm}
\begin{proof}
We argue by induction on $n$. When $n = 0$, the conclusion holds by assumption.
When $n = 1$, we note that the conclusion is obvious if $H^0(\oo_{\pp^r}(m)) \to H^0(\oo_S(m))$
is injective or if $H^0(\oo_{\pp^r}(m)) \to H^0(\oo_{S \cup T}(m))$
is surjective.
We may therefore suppose that the map $H^0(\oo_{\pp^r}(m)) \to H^0(\oo_S(m))$
is surjective but not injective, whose kernel contains a nonzero polynomial $f$;
and that $H^0(\oo_{\pp^r}(m)) \to H^0(\oo_{S \cup T}(m))$ is injective.
In particular, there is a point $p \in T$ with $f|_p \neq 0$.
Taking $T' = \{p\}$,
the map
$H^0(\oo_{\pp^r}(m)) \to H^0(\oo_{S \cup T'}(m))$
is surjective by construction.

For the inductive step, let $T'' \subset T$ be of size $n - 1$
such that
$H^0(\oo_{\pp^r}(m)) \to H^0(\oo_{S \cup T''}(m))$
is of maximal rank.
Applying our inductive hypothesis with $(S, T) = (S \cup T'', T \smallsetminus T'')$
completes the proof.
\end{proof}

\section{\boldmath The Case $r = 3$ \label{r}}

In this section, we will prove Theorems~\ref{add3} and~\ref{add4}.
As a consequence of Theorem~\ref{add3}, we will deduce
that if $C_1, C_2, \ldots, C_n \subset \pp^3$ are independently
general BN-curves, then
\[H^0(\oo_H(m)) \to H^0(\oo_{(C_1 \cup C_2 \cup \cdots \cup C_n) \cap H}(m))\]
is of maximal rank, unless $n = 1$, and $C_1$ is a canonically embedded
curve of genus $4$, and $m = 2$.
(In which case by inspection the above map fails
to be of maximal rank.)

\begin{proof}[Proof of Theorem~\ref{add3}]
If $C$ is not a canonical curve, Theorem~1.5 of \cite{quadrics}
states that $C \cap H$ is a general set of points,
and so Lemma~\ref{addsub} yields the desired result.

If $C$ is a canonical curve,
Theorem~1.5 of \cite{quadrics} states that $C \cap H$ is a set of $6$ points
which are general subject to the constraint that they lie on a conic.
In particular, $C \cap H$ imposes indepedent conditions on
$H^0(\oo_H(1))$ and on any fixed
proper subspace of $H^0(\oo_H(2))$.
Since $X$ is nonempty by assumption if $m = 2$, the kernel
of $H^0(\oo_H(m)) \to H^0(\oo_X(m))$
is a proper subspace of $H^0(\oo_H(m))$ if $m = 2$.
If $m \leq 2$, we therefore conclude that
\[H^0(\oo_H(m)) \to H^0(\oo_{X \cup (C \cap H)}(m))\]
is injective (so in particular of maximal rank as desired).

If $m \geq 3$,
we specialize the conic to the union of two lines,
and the points of $C \cap H$ to consist of $2$ points on one line
(which is just a set of $2$ general points),
and $4$ points on the other.
Using our assumption that
$H^0(\oo_H(m)) \to H^0(\oo_X(m))$
is of maximal rank and
applying Lemma~\ref{addsub} twice, it suffices
to show
\[H^0(\oo_H(m)) \to H^0(\oo_{X \cup Y}(m))\]
is of maximal rank, where $Y$ is a set of $\max(4, \dim \ker H^0(\oo_\Lambda(m)) \to H^0(\oo_{X \cap \Lambda}(m)))$ points
which are general subject to the condition that they lie on a line $\Lambda$.
For this, we use the exact sequence
\[0 \to \mathcal{I}_X(m - 1) \to \mathcal{I}_{X \cup Y}(m) \to \mathcal{I}_{Y/\Lambda}(m) \to 0.\]
Note that $H^0(\mathcal{I}_{Y/\Lambda}(m)) = 0$,
and if
$\dim \ker H^0(\oo_\Lambda(m)) \to H^0(\oo_{X \cap \Lambda}(m)) \geq 4$,
then we have $H^1(\mathcal{I}_{Y/\Lambda}(m)) = 0$ too.
In particular, the associated long exact sequence in cohomology
implies $H^0(\mathcal{I}_{X \cup Y}(m)) = 0$ provided that
$H^0(\mathcal{I}_X(m - 1)) = 0$,
and similarly for $H^1$ if
$\dim \ker H^0(\oo_\Lambda(m)) \to H^0(\oo_{X \cap \Lambda}(m)) \geq 4$.

Our assumption that
$H^0(\oo_H(m - 1)) \to H^0(\oo_X(m - 1))$
is of maximal rank
and injective if
$\dim \ker H^0(\oo_\Lambda(m)) \to H^0(\oo_{X \cap \Lambda}(m)) < 4$
thus implies that
$H^0(\oo_H(m)) \to H^0(\oo_{X \cup Y}(m))$
is of maximal rank, as desired.
\end{proof}

\begin{proof}[Proof of Theorem~\ref{add4}]
If $(d, g) \notin \{(8, 5), (9, 6), (10, 7)\}$,
Theorem~1.6 of \cite{quadrics}
states that $C \cap H$ is a general set of points,
and so Lemma~\ref{addsub} yields the desired result.

If $(d, g) \in \{(8, 5), (9, 6), (10, 7)\}$, then
Theorem~1.5 of \cite{quadrics} states that $C \cap H$ is a
general complete intersection of $3$ quadrics, a general set of $9$ points
on a complete intersection of $2$ quadrics, or a general set of $10$
points on a quadric, respectively.
In particular, we may specialize $C \cap H$ to consist of $8$ points
which are a general complete intersection of $3$ quadrics,
together with $d - 8$ independently general points.
Applying Lemma~\ref{addsub}, it suffices to show the result
when $(d, g) = (8, 5)$ and $C \cap H$ is a general complete intersection
of $3$ quadrics.

In particular $C \cap H$ imposes indepedent conditions on
$H^0(\oo_H(1))$ and on any fixed
subspace of $H^0(\oo_H(2))$ of codimension at least $3$.
Since any subscheme of $\pp^3$ of positive dimension or of degree at least $3$
imposes at least $3$ conditions on
quadrics, if $m \leq 2$
we therefore conclude that
\[H^0(\oo_H(m)) \to H^0(\oo_{X \cup (C \cap H)}(m))\]
is injective (so in particular of maximal rank as desired).

If $m \geq 3$,
we claim we may
further specialize $C \cap H$
to $8$ general points in a plane.
To see this, take a general set $\Gamma$ of $8$ points in a plane.
Then there is a smooth plane cubic curve $E$ containing $\Gamma$.
Let $p \in E$ be a point so that $\oo_E(2)(2p) \simeq \oo_E(\Gamma)$.
Choose a basis $\langle f_1, f_2, f_3 \rangle$ for $H^0(\oo_E(1))$,
so that $E \subset \pp^2$ is embedded via $[f_1 : f_2 : f_3]$,
and let $f_0$ be an extension to a basis of $H^0(\oo_E(1)(p))$.
Then for $\lambda$ generic,
the image of $\Gamma$ in $\pp^3$
under $[\lambda f_0 : f_1 ; f_2 : f_3]$
is a set of $8$ points on the image of $E$,
with class twice the pullback to $E$
under this embedding of the hyperplane class in $\pp^3$ --- in particular,
as $E$ is the complete intersection of two quadrics and is projectively
normal, is a complete intersection of $3$ quadrics in $\pp^3$.
Specializing $\lambda \to 0$, we obtain the set $\Gamma$
of $8$ general points in the plane that we started with.

Using our assumption that
$H^0(\oo_H(m)) \to H^0(\oo_X(m))$
is of maximal rank and
applying Lemma~\ref{addsub}, it suffices
to show
\[H^0(\oo_H(m)) \to H^0(\oo_{X \cup Y}(m))\]
is of maximal rank, where $Y$ is a set of
$\max(8, \dim \ker H^0(\oo_\Lambda(m)) \to H^0(\oo_{X \cap \Lambda}(m)))$ points
which are general subject to the condition that they lie on a plane $\Lambda$.

Note that $H^0(\mathcal{I}_{Y/\Lambda}(m)) = 0$,
and if
$\dim \ker H^0(\oo_\Lambda(m)) \to H^0(\oo_{X \cap \Lambda}(m)) \geq 8$,
then we have $H^1(\mathcal{I}_{Y/\Lambda}(m)) = 0$ too.
In particular, the associated long exact sequence in cohomology
implies $H^0(\mathcal{I}_{X \cup Y}(m)) = 0$ provided that
$H^0(\mathcal{I}_X(m - 1)) = 0$,
and similarly for $H^1$ if
$\dim \ker H^0(\oo_\Lambda(m)) \to H^0(\oo_{X \cap \Lambda}(m)) \geq 8$.

Our assumption that
$H^0(\oo_H(m - 1)) \to H^0(\oo_X(m - 1))$
is of maximal rank
and injective if
$\dim \ker H^0(\oo_\Lambda(m)) \to H^0(\oo_{X \cap \Lambda}(m)) < 8$
thus implies that
$H^0(\oo_H(m)) \to H^0(\oo_{X \cup Y}(m))$
is of maximal rank, as desired.
\end{proof}

\begin{cor} \label{cor:r}
If $C_1, C_2, \ldots, C_n$ are independently general space BN-curves,
$H \subset \pp^3$ is a general hyperplane,
and $m$ is a positive integer,
then the restriction map
\[H^0(\oo_H(m)) \to H^0(\oo_{(C_1 \cup C_2 \cup \cdots \cup C_n) \cap H}(m))\]
is of maximal rank, except if $m = 2$ and $n = 1$ and $C_1$
is a canonical curve.
\end{cor}
\begin{proof}
Applying Theorem~\ref{add3}, we immediately see all
cases of this statement by induction (starting with $n = 0$ as our base case),
provided we check the case when
$m = 3$ and $n = 2$ and $C_1$ and $C_2$ are both canonical curves.
In this case, by Theorem~1.5 of~\cite{quadrics}
$(C_1 \cup C_2) \cap H$ is a collection of $12$
points which are general subject to the condition
that $6$ of them lie on conic $Q_1$ and the other $6$ lie on a conic $Q_2$;
we want to show such the general such subscheme does not lie on any cubics.

For this, we specialize one of the points on $Q_1$ to one of the
points of intersection $Q_1 \cap Q_2$,
and one the points on $Q_2$ to a differnt point of intersection  $Q_1 \cap Q_2$.
The resulting subscheme of degree $12$ meets $Q_1$ in $7$ points,
but a cubic not containing $Q_1$ can only meet $Q_1$ in $6$ points
by Bezout's theorem.
Any such cubic must therefore contain $Q_1$, and symmetrically $Q_2$.
But $Q_1 \cup Q_2$ is of degree $4$, so is contained in no cubics, as desired.
\end{proof}

\section{\boldmath The Case $m = 2$ \label{m}}

In this section, we will prove the hyperplane maximal
rank theorem when $m = 2$, and the curves $C_i$ are all nonspecial. We will begin by constructing
reducible curves with the following lemma,
to which we will apply the method of Hirschowitz
outlined in the introduction.

\begin{lm} \label{foo}
Let $H' \subset \pp^r$ be a hyperplane, and
$(d, g)$ be integers with $d \geq g + r$ and $g \geq 0$.
Assume $d_1$ and $d_2$ are nonnegative integers
with $d = d_1 + d_2$.
Then there exist curves
$X \subset \pp^r$ and 
$Y \subset H'$, of degrees $d_1$ and $d_2$ respectively,
both of which are either nonspecial BN-curves, rational normal curves,
or empty;
with $X \cap Y$ general, such that
$X \cup Y \subset \pp^r$ is a nondegenerate BN-curve
of genus $g$ with $H^1(N_{X \cup Y}) = 0$.
\end{lm}
\begin{proof}
We argue by induction on $d$ (which satisfies $d \geq r$).
For the base case, we take $d = r$, which forces $g = 0$.
We may then let $X$ and $Y$ be rational normal curves
of degrees $d_1$ and $d_2$ respectively,
meeting at one point; this gives 
a BN-curve with $H^1(N_{X \cup Y}) = 0$
by Lemma~\ref{glone}.

For the inductive step,
we assume $d \geq r + 1$; in particular,
if $d_1 \leq 1$, then $d_2 \geq r$.
Define $g' = \max(0, g - 1)$ and
\[(d_1', d_2') = \begin{cases}
(d_1 - 1, d_2) & \text{if $d_1 \geq 2$;} \\
(d_1, d_2 - 1) & \text{else.}
\end{cases}\]
By our inductive hypothesis, there exists
curves $X' \subset \pp^r$ and 
$Y' \subset H'$, of degrees $d_1'$ and $d_2'$ respectively,
both of which are either nonspecial BN-curves, rational normal curves,
or empty;
with $X' \cap Y'$ general, such that
$X' \cup Y' \subset \pp^r$ is a nondegenerate BN-curve
of genus $g'$ with $H^1(N_{X' \cup Y'}) = 0$.

If $d_1 \geq 2$ and $g = 0$, we take $X = X' \cup L$ for $L$ a general
$1$-secant line to $X'$, and $Y = Y'$; by Lemma~\ref{glone},
both $X$ and $X \cup Y$ are BN-curves, and $H^1(N_{X \cup Y}) = 0$.

Similarly if $d_1 \leq 1$ and $g = 0$ (respectively $g \geq 1$),
we take $X = X'$, and $Y = Y' \cup L$
for $L$ a general $1$-secant (respectively $2$-secant)
line to $Y'$; by Lemma~\ref{glone},
both $Y$ and $X \cup Y$ are BN-curves, and $H^1(N_{X \cup Y}) = 0$.

Finally, we consider the case $d_1 \geq 2$ and $g \geq 1$.
If $X'$ is nondegenerate, we take $X = X' \cup L$ for $L$ a general
$2$-secant line to $X$, and $Y = Y'$; by Lemma~\ref{glone},
both $X$ and $X \cup Y$ are BN-curves, and $H^1(N_{X \cup Y}) = 0$.
If $X'$ is degenerate, then since $X' \cup Y'$ is nondegenerate
by assumption, the general line $L$ meeting $X'$ and $Y'$
each once intersects $Y'$ in a point which is independantly
general from $X' \cap Y'$.
We then take
we take $X = X' \cup L$, and $Y = Y'$; again by Lemma~\ref{glone},
both $X$ and $X \cup Y$ are BN-curves, and $H^1(N_{X \cup Y}) = 0$.
\end{proof}

\noindent
Combining this with Lemma~\ref{obvious}, we obtain:

\begin{cor} \label{mns}
Let $C_1, C_2, \ldots, C_n \subset \pp^r$ be independantly general nonspecial BN-curves,
and $H' \subset \pp^r$ be a hyperplane.
Then we may specialize the $C_i$ to curves $X_i \cup Y_i$
such that $\sum \deg X_i$ and $\sum \deg Y_i$ are any two nonnegative
integers adding up to $\sum \deg C_i$;
and such that $X_1, X_2, \ldots, X_n \subset \pp^r$
and $Y_1, Y_2, \ldots, Y_n \subset H'$ are each sets of independantly general
BN-curves or rational normal curves.
\end{cor}

\begin{prop} \label{m2}
Let $C_1, C_2, \ldots, C_n \subset \pp^r$ be
independantly general BN-curves, and $H \subset \pp^r$
be a general hyperplane. Assume that $C_i$ is nonspecial for all $i$.
Then
\[H^0(\oo_H(2)) \to H^0(\oo_{(C_1 \cup C_2 \cup \cdots \cup C_n) \cap H}(2))\]
is of maximal rank.
\end{prop}

\begin{proof}
We use induction on $r$;
when $r = 3$, this is a consequence of Corollary~\ref{cor:r}.
For the inductive step, write $d = \sum \deg C_i$,
and let $(d_1, d_2)$ be nonnegative integers
with $d = d_1 + d_2$, such that
\begin{align*}
d_1 \geq r \tand d_2 \geq \binom{r}{2} &\qquad \text{if} \quad d \geq \binom{r + 1}{2}, \\
d_1 \leq r \tand d_2 \leq \binom{r}{2} &\qquad \text{if} \quad d \leq \binom{r + 1}{2}.
\end{align*}

Pick a hyperplane $H'$ transverse to $H$.
By Corollary~\ref{mns}, 
we may specialize the $C_i$ to curves $X_i \cup Y_i$
such that $\sum \deg X_i = d_1$ and $\sum \deg Y_i = d_2$;
and such that
\[X := X_1 \cup X_2 \cup \cdots \cup X_n \subset \pp^r \tand Y := Y_1 \cup Y_2 \cup \cdots \cup Y_n \subset H'\]
are each unions
of independantly general
BN-curves or rational normal curves.
Since the hyperplane section of a rational normal curve
is a general set of points,
our inductive hypothesis in combination with Lemma~\ref{addsub}
implies
\[H^0(\oo_{H \cap H'}(2)) \to H^0(\oo_{Y \cap H}(2))\]
is of maximal rank. Define 
\[i = \begin{cases}
0 & \text{if}\ d \geq \binom{r + 1}{2}, \\
1 & \text{if}\ d \leq \binom{r + 1}{2};
\end{cases}\]
so we want to show
\[H^i(\mathcal{I}_{(C_1 \cup C_2 \cup \cdots \cup C_n) \cap H / H}(2)) = 0,\]
and know by induction that
\[H^i(\mathcal{I}_{Y \cap H / (H \cap H')}(2)) = 0\]

By direct examination,
$H^i(\mathcal{I}_{X \cap H / H}(1)) = 0$.
Consequently, we may use the
exact sequence of sheaves
\[0 \to \mathcal{I}_{X \cap H / H}(1) \to \mathcal{I}_{(C_1 \cup C_2 \cup \cdots \cup C_n) \cap H / H}(2) \to \mathcal{I}_{Y \cap H / (H \cap H')}(2) \to 0,\]
which gives rise to the long exact sequence in cohomology
\[\cdots \to H^i(\mathcal{I}_{X \cap H / H}(1)) \to H^i(\mathcal{I}_{(C_1 \cup C_2 \cup \cdots \cup C_n) \cap H / H}(2)) \to H^i(\mathcal{I}_{Y \cap H / (H \cap H')}(2)) \to \cdots,\]
to conclude that
$H^i(\mathcal{I}_{(C_1 \cup C_2 \cup \cdots \cup C_n) \cap H / H}(2)) = 0$
as desired.
\end{proof}

\subsection{\boldmath The Condition $d \geq g + r$}

The condition $d \geq r$ is necessary; indeed
when $d < g + r$, the map
will sometimes fail to be of maximal rank,
as shown by the following proposition:

\begin{prop} \label{fail} Let $C \subset \pp^r$ be any curve of degree $d$
and genus $g$, with $d < g + r$ and $4d - 2g < r(r + 3)$.
Then the restriction map
\[H^0(\oo_H(2)) \to H^0(\oo_{C \cap H}(2))\]
fails to be of maximal rank.
\end{prop}
\begin{proof}
We compute
\[\dim H^0(\oo_{\pp^r}(2)) - \dim H^0(\oo_C(2)) = \binom{r + 2}{2} - (2d + 1 - g) = \frac{r(r + 3) - (4d - 2g)}{2} > 0,\]
and so $C$ lies on a quadric. Moreover,
we have
\begin{align*}
\dim H^0(\oo_{\pp^r}(2)) - \dim H^0(\oo_C(2)) &= \frac{r(r + 3) - (4d - 2g)}{2} \\
&= \binom{r + 1}{2} - d + (g + r - d) \\
&= \dim H^0(\oo_H(2)) - \dim H^0(\oo_{C \cap H}(2)) + (g + r - d) \\
&> \dim H^0(\oo_H(2)) - \dim H^0(\oo_{C \cap H}(2)).
\end{align*}

Now every quadric containing $C$ restricts to a quadric in $H$
containing $H \cap C$;
as $C$ is nondegenerate, this restriction has no kernel.
Consequently, there is a subspace of
$H^0(\oo_H(2))$ in the kernel of
$H^0(\oo_H(2)) \to H^0(\oo_{C \cap H}(2))$
which is of positive dimension that exceeds
$\dim H^0(\oo_H(2)) - \dim H^0(\oo_{C \cap H}(2))$.
In other words, $H^0(\oo_H(2)) \to H^0(\oo_{C \cap H}(2))$
is not of maximal rank.
\end{proof}

When $n = 1$, the cases
in Proposition~\ref{fail} are the only
cases in which the restriction map
$H^0(\oo_H(2)) \to H^0(\oo_{C \cap H}(2))$
fails to be of maximal rank.

Indeed, if $C$ is a general BN-curve with $d < g + r$,
then $C$ is linearly normal, i.e.\ $H^1(\mathcal{I}_C(1))$ vanishes.
Now consider the exact sequence of sheaves
\[0 \to \mathcal{I}_C(1) \to \oo_{\pp^r}(1) \oplus \mathcal{I}_C(2) \to \mathcal{I}_{C \cap H}(2) \to 0;\]
this induces a long exact sequence of cohomology groups:
\[\cdots \to H^0(\oo_{\pp^r}(1)) \oplus H^0(\mathcal{I}_C(2)) \to H^0(\mathcal{I}_{C \cap H}(2)) \to H^1(\mathcal{I}_C(1)) \to \cdots\]
It follows that
$H^0(\oo_{\pp^r}(1)) \oplus H^0(\mathcal{I}_C(2)) \to H^0(\mathcal{I}_{C \cap H}(2))$
is surjective, i.e.\ every quadric $Q \subset H$
containing $C\cap H$ is the intersection with $H$ of a quadric $\tilde{Q}\subset \pp^r$
containing $C$.
For $4d - 2g \geq r(r + 3)$, the maximal
rank conjecture for quadrics
(see \cite{qb} or \cite{jp})
implies that $C$ is not contained in any quadric,
and consequently that $C \cap H$ is not contained in any quadric.

\section{Construction of Reducible Curves \label{sec:glue}}

In this section, which is the heart of the proof,
we will construct examples
of reducible BN-curves $X \cup Y$ where $Y \subset H'$.
These reducible curves will be the essential
ingredient in applying the inductive method of
Hirschowitz in the following section to deduce
the hyperplane maximal rank theorem.


\begin{lm} \label{gluea}
Let $H' \subset \pp^r$ be a hyperplane, and
$(d, g)$ be integers with $\rho(d, g, r) \geq 0$ and $d \geq g + r - 2$.
Assume $d_1$ and $d_2$ are positive integers with
$d = d_1 + d_2$, that additionally satisfy:
\[d_1 \geq r + \max(0, g + r - d) \tand d_2 \geq r - 1.\]
Then there exist nonspecial BN-curves $X \subset \pp^r$
and $Y \subset H'$ of degrees $d_1$ and $d_2$ respectively,
with $X \cap Y$ general, such that
$X \cup Y \subset \pp^r$ is a BN-curve
of genus $g$ with $H^1(N_{X \cup Y}) = 0$.
\end{lm}
\begin{proof}
We will argue by induction on $d$ and $\rho(d, g, r)$.
Notice that our inequalities for $d_1$ and $d_2$
imply $d \geq 2r - 1$; for the base case, we consider when $d = 2r - 1$
or $\rho(d, g, r) = 0$.

If $d = 2r - 1$, we take $X$ to be a rational normal curve of degree $r$,
and $Y \subset H$ to be a rational normal of degree $r - 1$
that meets $X \cap H$ in $g + 1$ points.
(Note that as $\rho(2r - 1, g, r) \geq 0$, we have $g + 1 \leq r$.)
By inspection, $X \cup Y$ is of genus $g$;
as $\aut H$ acts $(r + 1)$-transitively on points in linear
general position, $X \cap Y$ is general.
Moreover, $X \cup Y$ is a BN-curve with $H^1(N_{X \cup Y}) = 0$
by Lemma~\ref{glone}.

If $\rho(d, g, r) = 0$ and $d \geq g + r - 2$, then
either $(d, g) = (2r, r + 1)$
or $(d, g) = (3r, 2r + 2)$.
In the case $(d, g) = (2r, r + 1)$, we take $X$ to be the union
of a rational normal curve $R$ of degree $r$ with a $2$-secant line $L$,
and $Y$ to be a rational normal curve of degree $r - 1$
passing through $X \cap H$.
Again, by inspection $X \cup Y$ is of genus $r + 1$;
as $\aut H$ acts $(r + 1)$-transitively on points in linear
general position, $X \cap Y$ is general.
To see that $X \cup Y$ is a BN-curve with $H^1(N_{X \cup Y}) = 0$,
we apply Lemma~\ref{glone}
to the decomposition $X \cup Y = (Y \cup L) \cup R$.

Now suppose that $(d, g) = (3r, 2r + 2)$.
If $d_2 = r - 1$, then we take $X = C \cup L$
to be the union of a canonical curve $C$ with
a general $1$-secant line $L$.
We take $Y$ to be the rational normal curve of degree $r - 1$
passing through $L \cap H'$ and through $r + 1$ points of $C \cap H'$.
By inspection $X \cup Y$ is of genus $2r + 2$.
To see that $X \cap Y$ is general, first note
that since
$\aut H$ acts $(r + 1)$-transitively on points in linear
general position, $C \cap Y$ is general;
moreover, $L \cap H$ is general with respect to $C$.
To see that $X \cup Y$ is a BN-curve, we apply
Lemma~\ref{glone} to the decomposition $X \cup Y = C \cup (L \cup Y)$,
while noting that $L \cup Y$ is the specialization of a rational
normal curve of degree $r$.
Moreover, by Lemma~\ref{glonered}, we have $H^1(N_{X \cup Y}) = 0$.

Otherwise, we have $d_2 \geq r$ and $d_1 \geq r + 2$;
in this case we take
$X = R_1 \cup L_0 \cup L_1 \cup N_1$ and
$Y = R_2 \cup L_2 \cup N_2$, where:
\begin{enumerate}
\item $R_1$ is a general rational normal curve of degree $r$.
\item $L_0$ is a general $2$-secant line to $R_1$.
\item $R_2$ is a general rational normal curve of degree $r - 1$
passing through all $r + 1$ points of $(R_1 \cup L_0) \cap H$.
\item \label{l1} $L_1$ is a general line meeting $R_1$ once and $L_0$ once.
\item \label{l2} $L_2$ is a general $2$-secant line to $R_2$, passing through $L_1 \cap H$.
\item $N_1$ is a general rational normal curve of degree $d_1 - r - 2$
meeting $L_1$ once and $R_1$ in $d_1 - r - 2$ points
(we take $N_1 = \emptyset$ if $d_1 = r + 2$).
\item $N_2$ is a general rational normal curve of degree $d_2 - r$
meeting $L_2$ once and $R_2$ in $d_2 - r$ points
(we take $N_2 = \emptyset$ if $d_2 = r$).
\end{enumerate}

In order for this to make sense, we need
conditions \ref{l1} and \ref{l2} to be consistent.
The consistency of \ref{l1} and \ref{l2}, as well as
the assertion that $X \cap Y$ is general, both follow
from the following two claims:
\begin{itemize}
\item $L_1 \cap H$ is general
relative to $(R_1 \cup L_0) \cap H$.
This follows from $L_1 \cap R_1$ being general
relative to $L_0$ and $R_1 \cap H$,
which in turn follows from the existence of
a rational normal curve of degree $r$ through a
general collection of $r + 3$ points.
\item The $2$-secant lines to $R_2$ sweep out $H$
as we vary $R_2$ over all rational normal
curves of degree $r - 1$ passing through
all $r + 1$ points of $(R_1 \cup L_0) \cap H$.
This follows from the observation that $R_2$
sweeps out $H$, which again follows
from the existence of
a rational normal curve of degree $r - 1$ through a
general collection of $r + 2$ points in $H'$.
\end{itemize}

By inspection, $X \cup Y$ is a curve of genus $g$
and $X$ and $Y$ are nonspecial. 
To show that $X \cup Y$ is a BN-curve, we apply Lemma~\ref{glone}
to the decomposition
\[X \cup Y = (L_0 \cup R_2) \cup R_1 \cup (L_1 \cup L_2 \cup N_1 \cup N_2).\]
Similarly, to show $H^1(N_{X \cup Y}) = 0$,
we apply Lemma~\ref{glone} and then Lemma~\ref{gltwo}
to the decomposition
\[X \cup Y = (L_0 \cup R_2) \cup R_1 \cup L_2 \cup N_1 \cup N_2 \cup L_1.\]
To apply Lemma~\ref{gltwo}, we
need to check that the tangent lines to $(L_0 \cup R_2) \cup R_1 \cup L_2 \cup N_1 \cup N_2$
at the points of intersection with $L_1$ do not all lie
in a plane. Since $L_1$ intersects $L_0$,
the only possible plane that could
contain all $3$ tangents is $\overline{L_0L_1}$.
But as this plane contains the two points
of intersection of $L_0$ with $R_1$ and a plane
can only intersect a rational normal curve
at $3$ points with multiplicity,
the tangent line to $R_1$ at $L_1 \cap R_1$
cannot be contained in this plane.
Consequently, we may apply Lemma~\ref{gltwo} as claimed.

For the inductive step, we have $d \geq 2r$ and
$\rho(d, g, r) > 0$.
We claim that these inequalities imply that
\begin{equation} \label{nonspec}
r + \max(0, g + r - d) + r - 1 < d = d_1 + d_2.
\end{equation}
Of course,
\[r + \max(0, g + r - d) + r - 1 = \max(2r - 1, 3r - 1 + g - d);\]
consequently, as $2r - 1 < 2r \leq d$,
it suffices to show $3r - 1 + g - d < d$, or equivalently
$g < 2d + 1 - 3r$. To see this, note that if
$g \geq 2d + 1 - 3r$, then we would have
\[-(r - 1)(d - 2r) = (r + 1)d - r(2d + 1 - 3r) - r(r + 1) \geq (r + 1)d - rg - r(r + 1) > 0,\]
which is a contradiction; thus, $g < 2d + 1 - 3r$, and so
\eqref{nonspec} holds.
Consequently, there exists $(d_1', d_2')$
either equal to $(d_1 - 1, d_2)$ or to $(d_1, d_2 - 1)$,
such that $d_1' \geq r + \max(0, g + r - d)$
and $d_2' \geq r - 1$. (Otherwise $d_1 - 1 < r + \max(0, g + r - d)$
and $d_2 - 1 < r - 1$, i.e.\ $d_1 \leq r + \max(0, g + r - d)$
and $d_2 \leq r - 1$; adding these contradicts \eqref{nonspec}.)

If we define $g' = \max(0, g - 1)$,
then $\max(0, g + r - d) = \max(0, g' + r - (d - 1))$.
Thus by the inductive hypothesis, there are BN-curves
$X' \subset \pp^r$
and $Y' \subset H'$ of degrees $d_1'$ and $d_2'$ respectively,
with $X' \cap Y'$ general, such that
$X' \cup Y' \subset \pp^r$ is a BN-curve
of genus $g'$ with $H^1(N_{X' \cup Y'}) = 0$.
To complete the inductive step, we take
\[(X, Y) = \begin{cases}
(X', Y' \cup L) & \text{if $d_1' = d_1$;} \\
(X' \cup L, Y') & \text{if $d_2' = d_2$;} \\
\end{cases} \twhere L = \begin{cases}
\text{a $1$-secant line} & \text{if $g' = g$;} \\
\text{a $2$-secant line} & \text{if $g' \neq g$.}
\end{cases}
\]
This satisfies the desired conclusion by Lemma~\ref{glone}.
\end{proof}

\begin{lm} \label{glue}
Let $H' \subset \pp^r$ be a hyperplane, and
$(d, g)$ be integers with $\rho(d, g, r) \geq 0$.
Assume $d_1$ and $d_2$ are positive integers with
$d = d_1 + d_2$, that additionally satisfy:
\[d_1 \geq r + \max(0, g + r - d) \tand d_2 \geq r - 1.\]
Then there exists BN-curves $X \subset \pp^r$
and $Y \subset H'$ of degrees $d_1$ and $d_2$ respectively,
with $X \cap Y$ general, such that
$X \cup Y \subset \pp^r$ is a BN-curve
of genus $g$ with $H^1(N_{X \cup Y}) = 0$.
Moreover, we can take $X$ to be nonspecial if
\begin{equation} \label{ceil}
d_2 \geq (r - 1) \cdot \left\lceil \frac{\max(0, g + r - d)}{2}\right\rceil.
\end{equation}
\end{lm}

\begin{proof}
We will argue by induction on $d$. When $d \geq g + r - 2$, we are done by Lemma~\ref{gluea}.
Thus we may assume that $d < g + r - 2$.
In particular, this implies that $d \geq 4r$, and that
$\max(0, g + r - d) = g + r - d$.
We claim that
\begin{equation}\label{spec}
r + \max(0, g + r - d) + r - 1 = 3r - 1 + g - d < d - 2(r - 2) = d_1 + d_2 - 2(r - 2).
\end{equation}
This is equivalent to $g < 2d + 5 - 5r$; to see this, note that if
$g \geq 2d + 5 - 5r$, then
\[-(r - 1)(d - 4r) - 2r = (r + 1)d - r(2d + 5 - 5r) - r(r + 1) \geq (r + 1)d - rg - r(r + 1) = 0,\]
which is a contradiction; thus, $g < 2d + 5 - 5r$, and so
\eqref{spec} holds.
Consequently, there exists $(d_1', d_2')$
either equal to $(d_1 - 1, d_2 - r + 1)$ or to $(d_1 - r, d_2)$,
such that
\[d_1' \geq r + \max(0, g + r - d) - 1 = r + \max(0, (g - r - 1) + r - (d - r))\]
and $d_2' \geq r - 1$.
(Otherwise $d_1 - r < r + \max(0, g + r - d) - 1$
and $d_2 - r + 1 < r - 1$, i.e.\ $d_1 - (r - 2) \leq r + \max(0, g + r - d)$
and $d_2 - (r - 2) \leq r - 1$; adding these contradicts \eqref{spec}.)

Thus by the inductive hypothesis, there are BN-curves
$X' \subset \pp^r$
and $Y' \subset H'$ of degrees $d_1'$ and $d_2'$ respectively,
with $X' \cap Y'$ general, such that
$X' \cup Y' \subset \pp^r$ is a BN-curve
of genus $g - r - 1$ with $H^1(N_{X' \cup Y'}) = 0$.
To complete the inductive step, we take
\[(X, Y) = \begin{cases}
(X' \cup L, Y' \cup R_2) & \text{if $d_1' = d_1 - 1$;} \\
(X' \cup R_1, Y') & \text{if $d_2' = d_2$.} \\
\end{cases}\]
Here, $R_1$ is a rational normal curve of degree $r$ that is
$(r + 2)$-secant to $X'$,
and $L$ is a $1$-secant line to $X'$, and
$R_2$ is a rational normal curve of degree $r - 1$
intersecting $Y'$ in $r + 1$ points and passing through $L \cap H$.

Tracing through the proof, we notice then when \eqref{ceil}
is satisfied, we add a $1$-secant line to $X$
at least as many times as we add an $(r + 2)$-secant rational
normal curve of degree $r$.
In particular, when \eqref{ceil} holds,
the curve $X$ we constructed is
nonspecial.
\end{proof}

\begin{lm} \label{above}
Let $H' \subset \pp^r$ be a hyperplane, and
$(d, g)$ be integers with $\rho(d, g, r) \geq 0$.
Write
\[d_1 = 1 + \max(0, g + r - d) \tand d_2 = d - d_1.\]
Then there exists a rational curve $X \subset \pp^r$,
and a BN-curve $Y \subset H'$, of degrees $d_1$ and $d_2$ respectively,
with $X \cap Y$ general, such that
$X \cup Y \subset \pp^r$ is a BN-curve
of genus $g$ with $H^1(N_{X \cup Y}) = 0$.
\end{lm}
\begin{proof}
We argue by induction on $\rho(d, g, r)$.

When $\rho(d, g, r) = 0$, then $(d, g, r) = (r(t + 1), (r + 1)t, r)$
for some nonnegative integer $t$;
so for $\rho(d, g, r) = 0$ we may argue by induction on $t$.
When $t = 0$, we let $X$ be a line, and $Y$ be a rational normal curve of degree $r - 1$
in $H'$, passing through $X \cap H'$;
by Lemma~\ref{glone}, the union $X \cup Y$ is 
is a BN-curve
with $H^1(N_{X \cup Y}) = 0$.
For the inductive step, we let $X' \subset \pp^r$
and $Y' \subset H'$ be of degrees $d_1 - 1$ and $d_2 - r + 1$ respectively,
with $X' \cap Y'$ general, such that
$X' \cup Y' \subset \pp^r$ is a BN-curve
of degree $d - r$ and genus $g - r - 1$ with $H^1(N_{X' \cup Y'}) = 0$.
We then pick a general point $p \in H'$, and let
\[X = X' \cup L \tand Y = Y' \cup R,\]
where $L$ is a $1$-secant line to $X'$ through $p$,
and $R$ is a rational normal curve of degree $r - 1$
which is $(r + 1)$-secant to $Y'$ and passes through $p$.
Applying Lemmas~\ref{glone} and~\ref{glonered},
we conclude that the union $X \cup Y$ is
is a BN-curve of genus $g$
with $H^1(N_{X \cup Y}) = 0$ as desired.

For the inductive step, we let
$X' \subset \pp^r$
and $Y' \subset H'$ be of degrees $d_1$ and $d_2 - 1$ respectively,
with $X' \cap Y'$ general, such that
$X' \cup Y' \subset \pp^r$ is a BN-curve
of degree $d - 1$ and
genus $g' := \max(0, g - 1)$ with $H^1(N_{X' \cup Y'}) = 0$.
We then take
\[X = X' \tand Y = Y' \cup L,\]
where $L$ is a line which is $1$-secant to $Y'$ if $g = g'$
and $2$-secant otherwise.
Applying Lemma~\ref{glone},
we conclude that the union $X \cup Y$ is
is a BN-curve of genus $g$
with $H^1(N_{X \cup Y}) = 0$ as desired.
\end{proof}

\begin{lm} \label{key}
Let $C_1, C_2, \ldots, C_n \subset \pp^r$ be independantly general BN-curves,
of degrees $d_i$ and genera $g_i$,
and $H, H' \subset \pp^r$ be transverse hyperplanes. Let $d'$ and $d''$
be nonnegative integers with
\[d' + d'' = \sum d_i \tand d' \geq r - 1 + \sum [1 + \max(0, g_i + r - d_i)].\]
Then we may specialize the $C_i$ to curves $C_i^\circ$
with $\sum \# (C_i^\circ \cap H \cap H') = d''$, so that
\begin{gather*}
C_1^\circ \cap H \cap H', C_2^\circ \cap H \cap H', \ldots, C_n^\circ \cap H \cap H' \subset H \cap H' \quad \text{and} \\
C_1^\circ \cap H \smallsetminus H', C_2^\circ \cap H \smallsetminus H', \ldots, C_n^\circ \cap H \smallsetminus H' \subset H
\end{gather*}
are sets of subsets of hyperplane sections of independantly general
BN-curves.

Moreover, we can assume the second of these sets is a set of subsets
of hyperplane sections of independantly general
nonspecial BN-curves if
\begin{equation} \label{ceil}
d'' \geq (r - 1) \cdot \sum \left\lceil \frac{\max(0, g_i + r - d_i)}{2}\right\rceil.
\end{equation}
\end{lm}
\begin{proof}
We first note that it suffices to consider the case where all $C_i$ are special.
Indeed, if $C_1, C_2, \ldots, C_m$ are special,
and $C_{m + 1}, C_{m + 2}, \ldots, C_n$ are nonspecial,
then the result for $C_1, C_2, \ldots, C_n$
follows from the result for $C_1, C_2, \ldots, C_m$
combined with Corollary~\ref{mns} for
$C_{m + 1}, C_{m + 2}, \ldots, C_n$.

We may thus suppose $C_i$ is special for all $i$.
In particular, for all $i$,
\begin{equation} \label{dimin}
d_i \geq d_i - [(r + 1) d_i - r g_i - r(r + 1)] = r + r(g_i + r - d_i).
\end{equation}

We now argue by induction on $n$.
When $n = 1$, this follows from Lemmas~\ref{glue} and~\ref{obvious}
if $d'' \geq r - 1$.
If $d'' \leq r - 2$, then by the uniform position principle, the points
of $C_1 \cap H$ are in linear general position.
We may therefore apply an automorphism of $H$ so that exactly $d''$
of these points lie in $H'$.

For the inductive step, note that Equation~\eqref{dimin} gives,
in combination with $g_n + r - d_n \geq 1$,
\[d_n - r - \max(0, g_n + r - d_n) \geq (r - 1) \cdot \left\lceil\frac{\max(0, g_n + r - d_n)}{2}\right \rceil \geq r - 1.\]
In particular, so long as
\[2r - 2 + \sum [1 + \max(0, g_i + r - d_i)] \leq d' \leq d - r + 1,\]
we may combine our inductive hypothesis (for $C_1, \ldots, C_{n - 1}$)
with Lemmas~\ref{glue} and~\ref{obvious} (for $C_n$) to deduce the result.

If $d' \geq d - r + 2$, the result follows from our inductive hypothesis
(for $C_1, \ldots, C_{n - 1}$); we do not specialize $C_n$.

Finally, if $r - 1 + \sum [1 + \max(0, g_i + r - d_i)] \leq d' \leq 2r - 2 + \sum [1 + \max(0, g_i + r - d_i)]$,
then upon rearrangement,
\[d' - [1 + \max(0, g_n + r - d_n)] \leq 2r - 2 + \sum_{i < n} [1 + \max(0, g_n + r - d_n)],\]
so by combining Lemmas~\ref{above} and~\ref{obvious}(for $C_n$)
with
our inductive hypothesis (for $C_1, \ldots, C_{n - 1}$),
it suffices to show
$2r - 2 + \sum_{i < n} [1 + \max(0, g_n + r - d_n)] \leq \sum_{i < n} d_i$,
which follows in turn from
\[d_1 \geq 2r - 1 + \max(0, g_1 + r - d_1),\]
which in turn follows from Equation~\eqref{dimin}
together with $g_n + r - d_n \geq 1$.
\end{proof}

\section{The Inductive Argument \label{sec:ind}}

In this section, we combine the results of
the previous three sections to inductively
prove the hyperplane maximal rank theorem.
This essentially boils down to manipulating inequalities
to show that we can choose the integers $(d', d'')$ appearing in the previous
section in the appropriate fashion.

We begin by giving some bounds on the expressions appearing
in Lemma~\ref{glue} that are easier to manipulate.

\begin{lm} \label{ione}
Let $d$, $g$, and $r$ be integers with $\rho(d, g, r) \geq 0$. Then
\[1 + \max(0, g + r - d) \leq \frac{d}{r} \tand (r - 1) \cdot \left\lceil \frac{\max(0, g + r - d)}{2}\right\rceil \leq \frac{r - 1}{2r} \cdot d.\]
\end{lm}
\begin{proof}
By assumption,
\[r \cdot (g + r - d) \leq r \cdot (g + r - d) + (r + 1)d - rg - r(r + 1) = d - r\]
\[\Rightarrow \max(0, g + r - d) \leq \frac{d - r}{r}.\]
Substituting this in, we find
\begin{align*}
1 + \max(0, g + r - d) &\leq 1 + \frac{d - r}{r} = \frac{d}{r} \\
\left\lceil \frac{\max(0, g + r - d)}{2}\right\rceil &\leq \frac{\frac{d - r}{r} + 1}{2} = \frac{r - 1}{2r} \cdot d. && \qedhere
\end{align*}
\end{proof}

\begin{lm} \label{itwo}
Let $d$, $r$, and $m$ be integers with
\[r \geq 4, \quad m \geq 3, \tand d \geq 2r + 2.\]
Assume that
\[d \geq \binom{m + r - 1}{m}, \quad \text{respectively} \quad d \leq \binom{m + r - 1}{m}.\]
Then there are integers $d'$ and $d''$ such that $d = d' + d''$ and
\[d' \geq \binom{m + r - 2}{m - 1} \tand d'' \geq \binom{m + r - 2}{m},\]
\[\text{respectively} \quad d' \leq \binom{m + r - 2}{m - 1} \quad \text{and} \quad d'' \leq \binom{m + r - 2}{m},\]
which moreover satisfy
\[d' \geq r - 1 + \frac{d}{r} \tand d'' \geq r - 1.\]
Additionally, if $m = 3$, we can replace
$d'' \geq r - 1$ by the stronger assumption that
\[d'' \geq \frac{r - 1}{2r} \cdot d.\]
\end{lm}
\begin{proof}
First we consider the case where
\[d = \binom{m + r - 1}{m} \geq \binom{r + 2}{3} \geq 2r + 2.\]
In this case, we take
\[d' = \binom{m + r - 2}{m - 1} \tand d'' = \binom{m + r - 2}{m}.\]
To see that these satisfy the given conditions, first note that
\[d'' \geq \binom{r + 1}{3} \geq r - 1.\]
Next note that
\[\binom{m + r - 1}{m} \geq \frac{r - 1}{\frac{m}{m + r - 1} - \frac{1}{r}};\]
indeed, the LHS is an increasing function of $m$, the RHS is a decreasing function of $m$, and the inequality is obvious for $m = 3$. Rearranging, we get
\[d' = \binom{m + r - 2}{m - 1} = \frac{m}{m + r - 1} \cdot \binom{m + r - 1}{m} \geq r - 1 + \frac{1}{r} \cdot \binom{m + r - 1}{m} = r - 1 + \frac{d}{r}.\]
If $m = 3$, then
\[d'' = \binom{r + 1}{3} \geq \frac{r - 1}{2r} \cdot \binom{r + 2}{3} = \frac{r - 1}{2r} \cdot d.\]

In general, we induct upwards on $d$ in the $\geq$
case and downwards on $d$ in the $\leq$ case.
To do this, we want to show that if $d'$ and $d''$
satisfy
\[d' \geq r - 1 + \frac{d}{r} \tand d'' \geq \frac{r - 1}{2r} \cdot d \twhere d = d' + d'' \geq 2r + 2,\]
then either $(d' - 1, d'')$ or $(d', d'' - 1)$,
as well as either $(d' + 1, d'')$ or $(d', d'' + 1)$,
satisfy the above two conditions.
We note that
\begin{align*}
d' \geq r - 1 + \frac{d}{r} = r - 1 + \frac{d' + d''}{r} \quad &\Leftrightarrow \quad (r - 1) d' \geq r(r - 1) + d''. \\
d'' \geq \frac{r - 1}{2r} \cdot d = \frac{r - 1}{2r} \cdot (d' + d'') \quad &\Leftrightarrow \quad (r + 1) d'' \geq (r - 1) d'.
\end{align*}

Assume (to the contrary) that neither $(d' - 1, d'')$
nor $(d', d'' - 1)$ satisfy the conditions,
respectively that neither $(d' + 1, d'')$ 
nor $(d', d'' + 1)$ satisfy the conditions.
Then we must have
\begin{align*}
(r - 1) (d' - 1) < r(r - 1) + d'' &\tand (r + 1)(d'' - 1) < (r - 1) d', \\
\text{respectively} \quad (r - 1) d' < r(r - 1) + d'' + 1 &\tand (r + 1)d'' < (r - 1) (d' + 1).
\end{align*}
Equivalently, we must have
\begin{align*}
(r - 1) (d' - 1) + 1 \leq r(r - 1) + d'' &\tand (r + 1)(d'' - 1) + 1 \leq (r - 1) d', \\
\text{respectively} \quad (r - 1) d' \leq r(r - 1) + d'' &\tand (r + 1)d'' + 1 \leq (r - 1) (d' + 1).
\end{align*}
Adding twice the first equation to the second, we must have
\begin{align*}
2(r - 1) (d' - 1) + 2 + (r + 1)(d'' - 1) + 1 &\leq 2r(r - 1) + 2d'' + (r - 1) d', \\
\text{respectively} \quad 2(r - 1) d' + (r + 1)d'' + 1 &\leq 2r(r - 1) + 2d'' + (r - 1) (d' + 1).
\end{align*}
Simplifying yields
\[(r - 1)(d' + d'') \leq 2r^2 + r - 4, \quad \text{respectively} \quad (r - 1) (d' + d'') \leq 2r^2 - r - 2.\]
In particular,
\[d = d' + d'' \leq \frac{2r^2 + r - 4}{r - 1} = 2r + 3 - \frac{1}{r - 1} \quad \Rightarrow \quad d \leq 2r + 2.\]
Consequently, we can reach via upward and downward induction every
value of $d$ that is at least $2r + 2$.
\end{proof}

\begin{proof}[Proof of the Hyperplane Maximal Rank Theorem.]
We use induction on $m$ and $r$. For $m = 2$, this is a consequence of
Proposition~\ref{m2}; for $r = 3$, this is a consequence of
Corollary~\ref{cor:r}.
Note that if $\sum d_i \leq 2r - 1$, then all the $C_i$ are nonspecial and
so $H^0(\oo_H(2)) \to H^0(\oo_{C \cap H}(2))$ is surjective;
consequently, $H^0(\oo_H(m)) \to H^0(\oo_{C \cap H}(m))$
is surjective for all $m \geq 2$.
Thus, we may suppose $\sum d_i \geq 2r$.

For the inductive step, we define integers $(d', d'')$
as follows.
If $\sum d_i \in \{2r, 2r + 1\}$, we take $(d', d'') = (r + 1, \sum d_i - r - 1)$.
Otherwise, for $\sum d_i \geq 2r + 2$, we let $(d', d'')$
be as in Lemma~\ref{itwo}. Fix another hyperplane $H'$ transverse to $H$.
By Lemma~\ref{key}, plus Lemma~\ref{ione} when $d \geq 2r + 2$,
we may specialize the $C_i$ to curves $C_i^\circ$
with $\sum \# (C_i^\circ \cap H \cap H') = d''$, so that
\begin{gather*}
X := (C_1^\circ \cap H \cap H') \cup (C_2^\circ \cap H \cap H') \cup  \cdots \cup (C_n^\circ \cap H \cap H') \subset H \cap H' \quad \text{and} \\
Y := (C_1^\circ \cap H \smallsetminus H') \cup (C_2^\circ \cap H \smallsetminus H') \cup \cdots \cup (C_n^\circ \cap H \smallsetminus H') \subset H
\end{gather*}
are unions of subsets of hyperplane sections of independantly general
BN-curves.
Moreover, if $m = 3$, then we can arrange for $X$
to be a union of subsets of hyperplane sections of independantly general
nonspecial BN-curves.
By our inductive hypothesis, Lemma~\ref{addsub}, and the uniform position principle,
we know that the restriction maps
\[H^0(\oo_H(m - 1)) \to H^0(\oo_Y(m - 1)) \tand H^0(\oo_{H \cap H'}(m)) \to H^0(\oo_X(m))\]
are of maximal rank.

Define
\[i = \begin{cases}
0 & \text{if}\ \sum d_i \geq \binom{r + m - 1}{m}, \\
1 & \text{if}\ \sum d_i \leq \binom{r + m - 1}{m};
\end{cases}\]
so we want to show $H^i(\mathcal{I}_{(X \cup Y) \cap H}(m)) = 0$.
The exact sequence of sheaves
\[0 \to \mathcal{I}_{(X \cap H) / H}(m - 1) \to \mathcal{I}_{(X \cup Y) \cap H / H}(m) \to \mathcal{I}_{(Y \cap H)/(H \cap H')}(m) \to 0,\]
gives rise to a long exact sequence in cohomology
\[\cdots \to H^i(\mathcal{I}_{(X \cap H) / H}(m - 1)) \to H^i(\mathcal{I}_{(X \cup Y) \cap H / H}(m)) \to H^i(\mathcal{I}_{(Y \cap H)/(H \cap H')}(m)) \to \cdots.\]
By the inductive hypothesis, we have
$H^i(\mathcal{I}_{(X \cap H) / H}(m - 1)) = H^i(\mathcal{I}_{(Y \cap H)/(H \cap H')}(m)) = 0$.
Consequently, $H^i(\mathcal{I}_{(X \cup Y) \cap H / H}(m)) = 0$, as desired.
\end{proof}

\bibliography{mrcbib}{}

\begin{thebibliography}{1}

\bibitem{hh}
R.~Hartshorne and A.~Hirschowitz.
\newblock Smoothing algebraic space curves.
\newblock In {\em Algebraic geometry, {S}itges ({B}arcelona), 1983}, volume
  1124 of {\em Lecture Notes in Math.}, pages 98--131. Springer, Berlin, 1985.

\bibitem{mrat}
A.~Hirschowitz.
\newblock Sur la postulation g\'en\'erique des courbes rationnelles.
\newblock {\em Acta Math.}, 146(3-4):209--230, 1981.

\bibitem{jp}
David Jensen and Sam Payne.
\newblock Tropical independence {II}: {T}he maximal rank conjecture for
  quadrics.
\newblock {\em Algebra Number Theory}, 10(8):1601--1640, 2016.

\bibitem{over}
Eric Larson.
\newblock Degenerations of curves in projective space and the maximal rank
  conjecture.
\newblock \url{https://arxiv.org/abs/1809.05980}.

\bibitem{quadrics}
Eric Larson.
\newblock The generality of a section of a curve.
\newblock \url{http://arxiv.org/abs/1605.06185}.

\bibitem{mrc}
Eric Larson.
\newblock The maximal rank conjecture.
\newblock \url{https://arxiv.org/abs/1711.04906}.

\end{thebibliography}
\bibliographystyle{plain}

\end{document}